\documentclass{article}
\usepackage{amsthm,amsmath,amssymb,amsfonts,stmaryrd}
  \newtheorem{thm}{Theorem}[section]
  \newtheorem{lem}[thm]{Lemma}
  \newtheorem{prop}[thm]{Proposition}
  \newtheorem{cor}[thm]{Corollary}
  \theoremstyle{definition}
  \newtheorem{defn}[thm]{Definition}
  \newtheorem{exm}[thm]{Example}
  \newtheorem{rmk}[thm]{Remark}
  
 \newcommand\ra{\rightarrow}

 \newcommand\s{\subseteq}

 \numberwithin{equation}{section}

\begin{document}

\title{\bf Orthocomplete Pseudo $MV$-algebras }
\author{\bf Anatolij Dvure\v{c}enskij$^{^{1,2}}$, Omid Zahiri$^{^{3}}$\footnote{Corresponding author }\\
 {\small\em $^1$ Mathematical Institute,  Slovak Academy of Sciences, \v Stef\'anikova 49, SK-814 73 Bratislava, Slovakia} \\
{\small\em $^2$ Depart. Algebra  Geom.,  Palack\'{y} Univer.17. listopadu 12,
CZ-771 46 Olomouc, Czech Republic} \\
{\small\em  $^3$University of Applied Science and Technology, Tehran, Iran}\\
{\small\em dvurecen@mat.savba.sk\quad  om.zahiri@gmail.com} \\
}
\date{}
\maketitle
\begin{abstract}
Pseudo $MV$-algebras are a non-commutative generalization of $MV$-algebras.
The main purpose of the paper is to introduce and investigate orthocomplete pseudo $MV$-algebras. We use the concepts of projectable pseudo $MV$-algebras and large pseudo $MV$-subalgebras to introduce orthocomplete pseudo $MV$-algebras. Then we apply a generalization of the Mundici's functor to an orthocompletion of an representable $\ell$-group to prove that each representable pseudo $MV$-algebra has an orthocompletion. In particular, our results are valid also for $MV$-algebras.
\end{abstract}



{\small {\it Acknowledgement:} This work was supported by  the Slovak Research and
Development Agency under contract APVV-0178-11 and  grant VEGA No. 2/0059/12 SAV, and
GA\v{C}R 15-15286S.}

\section{ Introduction }

In \cite{Bernau},  Bernau  introduced the concept of the orthocompletion of a lattice-ordered group (= an $\ell$-group) and proved that each representable
$\ell$-group has a unique orthocompletion. The definition was clarified in \cite{Conrad1}, where a simpler construction of the orthocompletion was given. It was shown that this construction is essentially a direct limit of cardinal products of quotients by polars.
Ball \cite{Ball}  has generalized these notions to the non-representable case. Another construction for the orthocompletion of $\ell$-groups can be found in \cite{Bleier}.

$MV$-algebras were defined by Chang, \cite{Cha}, as an algebraic counterpart of many-valued reasoning. The principal result of theory of $MV$-algebras is a representation theorem by Mundici \cite{Mun} saying that there is a categorical equivalence between the category of MV-algebras and the category of unital Abelian $\ell$-groups. Today theory of $MV$-algebras is very deep and has many interesting connections with other parts of mathematics with many important applications to different areas. For more details on $MV$-algebras, we recommend the monograph \cite{mv-book}.

In the last period, there appeared also two non-commutative generalizations of $MV$-algebras, called pseudo $MV$-algebras by \cite{Georgescu}, or equivalently, $GMV$-algebras by \cite{Rac}. They can be represented as an interval in  unital $\ell$-groups (not necessarily Abelian) as it follows from the fundamental result of Dvure\v{c}enskij \cite{151} which generalizes the Mundici representation theorem of $MV$-algebras.

Non-commutative operations, for example multiplication of matrices, are well known both in mathematics and physics and their applications. In particular, the class of square  matrices of the form
$$ A(a,b)=
\left( \begin{array}{cc}
a & b \\
0& 1
\end{array}
\right)
$$
for $a>0$, $b \in (-\infty,\infty)$ with usual multiplication of matrices is a non-commutative linearly ordered group with the neutral element $A(1,0)$ and with the positive cone consisting of matrices $A(a,b)$ with $a>1$ or $a=1$ and $b\ge0$. It gives an example of a pseudo $MV$-algebra. We note that $A(a,b)$ is an extension of real numbers: If $b=0$, then $A(a,0)$ is a positive real number and if $b\ne 0$, then $A(a,b)$ denotes some kind of a generalized number (non-standard number) such that $A(a,b)$ is infinitely close to $A(a,0)$ but bigger than $A(a,0)$, and similarly if $b<0$, then $A(a,b)$ is also infinitely closed to $A(a,0)$ but smaller than $A(a,0)$, \cite{Haj}.

Pseudo $MV$-algebras generalize $MV$-algebras, and in contrast to $MV$-algebras, not every pseudo $MV$-algebra is a subdirect product of linearly ordered pseudo $MV$-algebras. Pseudo $MV$-algebras are an algebraic counterpart of non-commutative reasoning.  Representable pseudo $MV$-algebras are those that are a subdirect product of linearly ordered pseudo $MV$-algebras. In \cite{DvuS}, it was shown that the class of representable pseudo $MV$-algebras is a variety. Since a pseudo $MV$-algebra is linearly ordered iff its representing unital $\ell$-group is linearly ordered, every representable pseudo $MV$-algebra is in a one-to-one correspondence with representable unital $\ell$-groups.

In \cite{Jakubik}, Jakub\' ik defined projectable $MV$-algebras, retracts and retract mappings and investigate
the relation between retract mappings of a projectable $MV$-algebra and retract mappings of its corresponding lattice ordered group. In \cite{Jak2001}, he studied a direct product decomposition of pseudo $MV$-algebras.

Recently, Ledda et al. \cite{LPT}, considered the projectability property as a lattice-theoretic property for more general classes of algebras of logic. For a class of integral residuated lattices that includes Heyting algebras and representable residuated lattices, they proved that an algebra of such a class is projectable iff the order dual of each subinterval $[a, 1]$ is a Stone lattice. In particular, they showed that a pseudo $MV$-algebra is projectable iff its bounded lattice reduct can be endowed with a G\"{o}del implication. 

In our contribution, we continue in this research studying projectable pseudo $MV$-algebra. Our aim is to describe the orthocompletion of representable pseudo $MV$-algebras in an analogy with orthocompletion of $\ell$-groups.

In the present paper, we introduce summand-ideals of pseudo $MV$-algebras in order to study orthocomplete pseudo $MV$-algebras. We present a relation between an essential extension and an orthocompletion of a representable pseudo $MV$-algebra $A$ and its representation unital $\ell$-group $(G_A,u_A)$. We show that if $A$ is a large subalgebra of a pseudo $MV$-algebra $B$, then the intersection of all projectable pseudo $MV$-subalgebras of $B$ containing $A$ is a projectable pseudo $MV$-algebra. Then we use the orthocompletion of representable unital $\ell$-groups for representable pseudo $MV$-algebras in order to show that any representable pseudo $MV$-algebras have an orthocompletion. Finally, we give some results and use the orthocompletion of a representable pseudo $MV$-algebra $A$ to obtain a minimal strongly projectable essential extension for the pseudo $MV$-algebra $A$, which is the intersection of all projectable pseudo $MV$-subalgebra of $O(A)$ (the orthocompletion of $A$) containing $A$.

\section{ Preliminaries}

In the section, we gather some basic notions relevant to $MV$-algebras and
$\ell$-groups which will be needed in the next section. For more details, we recommend to consult papers \cite{Anderson, Darnel} for theory of $\ell$-groups and \cite{mv-book, Georgescu} for $MV$-algebras and pseudo $MV$-algebras.

\begin{defn}\cite{Anderson}
A group $(G;+,0)$ is said to be {\it partially ordered} if it is equipped with a partial order relation $\leq$ which is compatible with $+$, that is,
$a\leq b$ implies that $x+a+y\leq x+b+y$ for all $x,y\in G$. An element $x\in G$ is called {\it positive} if $0\leq x$.
The partially ordered group $(G;+,0)$ is called a {\it lattice ordered group} or simply an $\ell$-{\it group}
if $G$ with its partially order relation is a lattice.
Any lattice ordered group satisfies the following properties (see \cite{Anderson, Darnel}):
\begin{itemize}
\item[($\ell$i)] $x+(y\vee z)=(x+y)\vee (x+z)$ and $x+(y\wedge z)=(x+y)\wedge (x+z)$;

\item[($\ell$ii)] $-(x\vee y)=-x\vee -y$ and $-(x\wedge y)=-x\wedge -y$;

\item[($\ell$iii)] for positive elements $x,y$ and $z$, $x\wedge(y+z)\leq (x\wedge y)+(x\wedge z)$.
\end{itemize}
\end{defn}

Let $(G;+,0)$ be an $\ell$-group. A normal convex $\ell$-subgroup of $G$ is called an $\ell$-{\it ideal}.
For each $g\in G$, the absolute value $|g|$ of $g$ is $g^{+}+g^{-}$, where $g^{+}=g\vee 0$ and $g^{-}=-g\vee 0$. The absolute
value satisfies a weakened triangle inequality:

(WTI) $|x+y|\leq |x|+|y|+|x|$.

We call a convex $\ell$-subgroup $C$ of $G$ satisfying the condition $C^{\bot\bot}=C$ a {\it polar subgroup} of $G$ and we
denote the collection of such by $\rho(G)$, where $C^{\bot_{G}}$ or simply $C^{\bot}$ is a unique maximal convex $\ell$-subgroup
for which $C\cap C^\bot=\{0\}$. It is clear that $\rho(G)$ is a Boolean algebra (see \cite{Anderson}).

We remind that an $\ell$-group $G$ is {\it representable} if it is a subdirect product of linearly ordered groups. Representable $\ell$-groups form a variety.

If an $\ell$-group $(G;+,0)$ is an $\ell$-subgroup of an $\ell$-group $(H;+,0)$, we write $G\leq H$.

An element $u$ of an $\ell$-group $(G;+,0)$ is called a {\it strong unit} (or an order unit) if, for each $x\in G$, there exists $n\in \mathbb{N}$ such
that $x\leq nu$. A {\it unital} $\ell$-group is a couple $(G,u)$, where $G$ is an $\ell$-group with a fixed strong unit $u$.

According to \cite{Georgescu}, a {\it pseudo MV-algebra} is an algebra $(M;
\oplus,^-,^\sim,0,1)$ of type $(2,1,1,$ $0,0)$ such that the
following axioms hold for all $x,y,z \in M$ with an additional
binary operation $\odot$ defined via $$ y \odot x =(x^- \oplus y^-)
^\sim $$
\begin{enumerate}
\item[{\rm (A1)}]  $x \oplus (y \oplus z) = (x \oplus y) \oplus z;$

\item[{\rm (A2)}] $x\oplus 0 = 0 \oplus x = x;$

\item[{\rm (A3)}] $x \oplus 1 = 1 \oplus x = 1;$

\item[{\rm (A4)}] $1^\sim = 0;$ $1^- = 0;$

\item[{\rm (A5)}] $(x^- \oplus y^-)^\sim = (x^\sim \oplus y^\sim)^-;$

\item[{\rm (A6)}] $x \oplus (x^\sim \odot y) = y \oplus (y^\sim
\odot x) = (x \odot y^-) \oplus y = (y \odot x^-) \oplus
x;$\footnote{$\odot$ has a higher priority than $\oplus$.}

\item[{\rm (A7)}] $x \odot (x^- \oplus y) = (x \oplus y^\sim)
\odot y;$

\item[{\rm (A8)}] $(x^-)^\sim= x.$
\end{enumerate}

For example, if $u$ is a strong unit of a (not necessarily Abelian)
$\ell$-group $G$,
$$\Gamma(G,u) := [0,u]
$$
and
\begin{eqnarray*}
x \oplus y &:=&
(x+y) \wedge u,\\
x^- &:=& u - x,\\
x^\sim &:=& -x +u,\\
x\odot y&:= &(x-u+y)\vee 0,
\end{eqnarray*}
then $(\Gamma(G,u);\oplus, ^-,^\sim,0,u)$ is a pseudo $MV$-algebra
\cite{Georgescu}.

(A6) defines the join $x\vee y$ and (A7) does the meet $x\wedge y.$ In addition, $M$ with respect to $\vee$ and $\wedge$ is a distributive lattice, \cite{Georgescu}.

A pseudo $MV$-algebra is an $MV$-{\it algebra} iff $x\oplus y=y\oplus x$ for all
$x,y \in M$. We note that if $x^-=x^\sim$ for each $x\in M$, $M$ is said to be {\it symmetric}. We note that a symmetric pseudo $MV$-algebra is not necessarily an $MV$-algebra.

	
	

In addition, let $x \in M$. For any integer $n\ge 0$, we set
$$0.x=0,\quad 1.x=1,\quad n.x=(n-1).x\oplus x, \ n\ge 2,$$
and
$$x^0=1,\quad x^1 =1, \quad x^n=x^{n-1}\odot x,\ n\ge 2.$$

An element $a\in M$ is called a {\it Boolean element} if $a\oplus a=a$, or equivalently, $a\odot a=a$; we denote by $B(M)$ the set of Boolean elements of $M$; it is a Boolean algebra that is a subalgebra of $M$. The following conditions are equivalent: (i) an element $a\in M$ is a Boolean element, (ii) $a\wedge a^- =0$, (iii) $a\wedge a^\sim= 0$, \cite[Prop. 4.2]{Georgescu}. In addition, if $a\in B(M)$, then $a^-=a^\sim$ and therefore, we put $a':=a^-$, and $a\oplus x = a\vee x = x\oplus a$ for each $x \in M$, see \cite[Prop. 4.3]{Georgescu}.

A non-empty subset $I$ of a pseudo $MV$-algebra $M$ is called an {\it ideal} of $M$ if $I$ is a down set which is closed under $\oplus$. An ideal $I$ of $M$ is said to be (i) {\it prime} if $x\wedge y \in I$ implies $x \in I$ or $y \in I$, and (ii) {\it normal} if $x\oplus I=I\oplus x$ for any $x \in M$, where $x\oplus I:=\{x\oplus i \mid i \in I\}$ and $I\oplus x =\{i\oplus x \mid i \in I\}$. Two equivalent conditions, \cite[Thm 2.17]{Georgescu}, to be an ideal $I$ prime are: (i) $x\odot y^- \in I$ or $y\odot x^-$ for all $x,y \in M$, (ii) $x\odot y^\sim \in I$ or $y\odot x^\sim$ for all $x,y \in M$.

If $X$ is a subset of a pseudo $MV$-algebra $M$, we denote (i) by $\langle X\rangle_n$ the normal ideal of $M$ generated by $X$, and (ii) a {\it polar} of $X$, i.e. the set $X^{\bot_M}:=\{y \in M\mid x\wedge y =0, \forall\, x \in X\}$, or simply we put $X^\bot:= X^{\bot_M}$. If $X=\{a\}$, we put $a^\bot:=\{a\}^\bot$.

An ideal $I$ of $M$ is called a {\it polar ideal}  if $I^{\bot_{M}\bot_{M}}=I$. It can be easily seen that a subset $I$ of $M$ is a polar ideal if and only if
$I=\{a\in M\mid a\wedge x=0, \forall x\in X\}$ for some subset $X$ of $M$. The set of polar ideals of $M$ is denoted by
$\rho(M)$.

There is a one-to-one relationship between congruences and normal ideals of a pseudo $MV$-algebra, \cite[Cor. 3.10]{Georgescu}: If $I$ is a normal ideal of a pseudo MV-algebra, then $x\sim_I y$ iff $x\odot y^-, y\odot x^-\in I$ is a congruence, and $M/I$ is a pseudo $MV$-algebra. Conversely, if $\sim$ is a congruence on $M$, then $I_\sim =\{x \in M \mid x\sim 0\}$ is a normal ideal such that $\sim_{I_\sim}=\sim$.

A pseudo $MV$-algebra $M$ is {\it representable} if $M$ is a subdirect product of a system of linearly ordered pseudo $MV$-algebras. By \cite[Thm 6.8]{DvuS}, the class of representable pseudo $MV$-algebras is a variety, and by \cite[Prop. 6.9]{DvuS}, $M$ is representable iff $a^\bot$ is a normal ideal of $M$ for each $a \in M$.

If a pseudo $MV$-algebra $(M;\oplus,^-,^\sim,0,1)$ is a subalgebra of a pseudo $MV$-subalgebra $(N;\oplus,^-,^\sim,0,1)$, we write $M\leq N$.

It is well known that according to Mundici, \cite{Mun}, there is a close connection between unital Abelian $\ell$-groups $(G,u)$ and $MV$-algebras. An analogous result holds for pseudo $MV$-algebras which was established by Dvure\v{c}enskij in \cite{151}. It says that, for each pseudo $MV$-algebra, there is unique (up to isomorphism) unital $\ell$-group $(G,u)$ such that $M \cong \Gamma(G,u)$. Moreover, there is a categorical equivalence between the category of unital $\ell$-groups and the category of pseudo MV-algebras: Let $\mathcal{UG}$ be the class of unital $\ell$-groups whose objects are unital $\ell$-groups $(G,u)$ and morphisms between objects are $\ell$-group homomorphisms preserving fixed strong units. We denote by $\mathcal{PMV}$ the class of pseudo MV-algebras whose objects are pseudo MV-algebras and morphisms are homomorphisms of pseudo MV-algebras. Then $\mathcal{UG}$ and $\mathcal{PMV}$ are categories. The generalized Mundici's functor $\Gamma:\mathcal{UG}\to \mathcal{PMV}$ is defined as follows
$$
\Gamma(G,u)=(\Gamma(G,u);\oplus, ^-,^\sim,0,u)
$$
and if $h:(G,u)\to (H,v)$ is a morphism, then

$$
\Gamma(h) = h_{|[0,u]}.
$$
It is important to note that $\mathcal{PMV}$ is a variety whereas $\mathcal {UG}$ not because it is not closed under infinite products.

Moreover, there is another functor from the category of pseudo $MV$-algebras to $\mathcal{UG}$ sending a pseudo $MV$-algebra $M$ to a unital $\ell$-group $(G,u)$ such that $M \cong \Gamma(G,u)$ which is
denoted by $\Xi:\mathcal{PMV}\ra \mathcal{UG}$. For more details relevant to these functors, please see \cite{151}.

\begin{thm}{\rm\cite{151}}\label{functor}
The composite functors $\Gamma\Xi$ and $\Xi\Gamma$ are naturally equivalent to the identity functors of
$\mathcal{PMV}$ and $\mathcal{UG}$, respectively. Therefore, $\mathcal{UG}$ and $\mathcal{PMV}$ are categorically equivalent.

In addition, if $h:\ \Gamma(G,u) \to \Gamma(G',u')$ is a
morphism of pseudo $MV$-algebras, then there is a unique homomorphism of unital $\ell$-groups
$f:\ (G,u) \to (G',u')$  such that $h =\Gamma(f)$, and
\begin{enumerate}
\item[{\rm (i)}] if $h$ is surjective, so is $f$;
 \item[{\rm (ii)}] if $h$ is  injective, so is $f$.
\end{enumerate}
\end{thm}

Therefore, a pseudo $MV$-algebra $\Gamma(G,u)$ is representable iff an $\ell$-group $G$ is representable.

A relation between some ideals of pseudo $MV$-algebras and some convex subgroups of unital $\ell$-groups is as follows; its $MV$-variant was established in \cite{Cignoli}:

\begin{thm}\label{ideals}{\rm \cite[Thm. 6.1]{DvuS}}
Let $(G,u)$ be a unital $\ell$-group. The map $\Phi: J\mapsto  \{x\in G\mid |x|\wedge u\in J\}$
defines an isomorphism from the poset of normal ideals of $\Gamma(G,u)$ onto the poset of $\ell$-ideals of $G$.
The inverse isomorphism is given by the map $\Psi:H\mapsto H\cap [0,u]$. Furthermore, let $Spec(G)$ be the set of all
proper prime $\ell$-ideals of $G$ and $Spec(\Gamma(G,u))$ be the set of prime ideals of $\Gamma(G,u)$. Then
$(Spec(\Gamma(G,u)),\s)\cong (Spec(G,u),\s)$. Moreover, the maps $\Phi$ and $\Psi$ define a one-to-one relations between ideals of $\Gamma(G,u)$ and convex subgroups of $G$.
\end{thm}

\section{ {\bf Summand-Ideals of Pseudo $MV$-algebras}} 

In the section we present summand-ideals and we show their close connection with polars.

From now on, unless otherwise stated, we will assume that $(M;\oplus,^-,^\sim,0,1)$ or simply $M$ is a pseudo $MV$-algebra and for each subsets $X$ and $Y$ of $M$,
$X\oplus Y=\{x\oplus y\mid (x,y)\in X\times Y\}$. We start with the following useful two lemmas.

\begin{lem}\label{le:1}
Let $A$ and $B$ be normal ideals of a pseudo $MV$-algebra $M$. Then

\begin{eqnarray*}
\langle A\cup B\rangle_n &=&\{x\in M \mid x\le a\oplus b \mbox{ for some } a\in A,\ b \in B\}\\
& =& \{x\in M \mid x= a\oplus b \mbox{ for some } a\in A,\ b \in B\}\\
&=& A \oplus B.
\end{eqnarray*}
\end{lem}

\begin{proof}
If we set $I= \{x\in M \mid x\le a\oplus b \mbox{ for some } a\in A,\ b \in B\}$, then $I$ contains $A$ and $B$. It is clear that it is a down set. Let $x,y \in I$, then $x\le a_1\oplus b_1$ and $y \le a_2\oplus b_2$ for some $a_1,a_2 \in A$ and $b_1,b_2\in B$. Then $x\oplus y \leq a_1 \oplus b_1 \oplus a_2 \oplus b_2= a_1 \oplus a_2 \oplus b_1' \oplus b_2$, where $b_1' \in B$. Hence, $I$ is an ideal. Using the Riesz decomposition property, \cite[Thm. 2.6]{Dvu1}, we have that if $x \in I$ and $x \le a\oplus b$, then there are $a_1 \le a$ and $b_1\le b$ such that $x = a_1 \oplus b_1$. Hence, $I = \{x\in M \mid x= a\oplus b \mbox{ for some } a\in A,\ b \in B\}$, in addition, $I = \{x\in M \mid x= b\oplus a \mbox{ for some } a\in A,\ b \in B\}$ which is true in view of normality of $A$ and $B$. Take $z\in M$ and $x =a\oplus b\in I$. Then $z\oplus x= z \oplus (a\oplus b)= a'\oplus b' \oplus z$ for some $a'\in A$ and $b' \in B$, and similarly we have $a\oplus b \oplus z = z \oplus a'' \oplus b''$ which proves that $I$ is a normal ideal of $M$ generated by $A\cup B$.
\end{proof}

\begin{lem}\label{le:2}
If $a$ is a Boolean element of a pseudo $MV$-algebra $M$, then $\downarrow a$ is a normal ideal of $M$.
\end{lem}

\begin{proof}
Let $a$ be a Boolean element of $M$. Then evidently $\downarrow a$ is an ideal of $M$. Let $x \in M$ and $b\in \downarrow a$. Then $x\oplus b = (x\oplus b)\odot x^- \oplus x$. Since $b \le a$, we have $(x\oplus b)\odot x^- \le (x\oplus a)\odot x^- =(x\vee a)-x = (x-x)\vee (a- x) \le 0\vee a=a$, where $-$ is the group subtraction taken from the corresponding unital $\ell$-group. In a similar way we can prove the second property of normality of $\downarrow a$.
\end{proof}

\begin{defn}\label{3.1}
A normal ideal $I$ of a pseudo $MV$-algebra $M$ is called a {\it summand-ideal} if there exists a normal ideal $J$ of $M$ such that $\langle I\cup J\rangle_n=M$ and $I\cap J=\{0\}$.
In this case, we write $M=I\boxplus J$. The set of all summand-ideals of $M$ is denoted by $\mathfrak{Sum}(M)$. Evidently, $\{0\}, M \in \mathfrak{Sum}(M)$.
\end{defn}

In the next proposition, we will gather some properties of summand-ideals of an $MV$-algebra.

\begin{prop}\label{3.2}
Let $A$ be a normal ideal of a pseudo $MV$-algebra $(M;\oplus,^-,^\sim,0,1)$.
\begin{itemize}
\item[{\rm (i)}] If $A$ is a summand-ideal of $M$ such that $M=A\boxplus B$ for some ideal $B$ of $M$, then $M=A\oplus B$ and  $B=A^{\bot}$.
\item[{\rm (ii)}] $A\in \mathfrak{Sum}(M)$ if and only if $M=A\oplus A^{\bot}$.
\item[{\rm (iii)}] If $A$ is a summand-ideal of $M$, then $I=(A\cap I)\oplus(I\cap A^\bot)$ for each normal ideal $I$ of $M$.
\item[{\rm (iv)}] If $A\in \mathfrak{Sum}(M)$, then $A=A^{\bot\bot}$. That is, $\mathfrak{Sum}(M)\s \rho(M)$.
\item[{\rm (v)}] If $A,B\in \mathfrak{Sum}(M)$, then $A\cap B\in \mathfrak{Sum}(M)$.
\item[{\rm (vi)}] For each $x\in M$, there exist unique elements $a\in A$ and $b\in A^\bot$ such that $x=a\oplus b$.
\item[{\rm (vii)}] If $A$ is a summand-ideal of $M$, then there is a unique element $a\in A\cap B(M)$ such that $A=\downarrow a$.
\end{itemize}
\end{prop}

\begin{proof}
(i) Let $A$ be a summand-ideal of $M$ such that $M=A\boxplus B$ for some normal ideal $B$ of $M$. Then $A\cap B=\{0\}$ and $M=\langle A\cup B\rangle_n$. If $b\in B$, then
$b\wedge a\in A\cap B$ for all $a\in A$, so $b\in A^{\bot}$. Thus $B\s A^{\bot}$. Moreover, if $x\in A^{\bot}$,
then by Lemma \ref{le:1}, there exist $a\in A$ and $b\in B$ such that $x= a\oplus b$.
Since $A$ and $B$ are ideals of $M$, then $a_1\in A$ and $b_1\in B$, so $x\in A\oplus B$ and
$$
0=x\wedge a_1=(a_1\oplus b_1)\wedge a_1\geq a_1\wedge a_1=a_1.
$$
It follows that $x=b_1\in B$. Therefore, $B=A^{\bot}$.

(ii) Let $A\in \mathfrak{Sum}(M)$. Then there exists a normal ideal $B$ of $M$ such that $M=A\boxplus B$ and so
by (i), $M=A\boxplus A^{\bot}$. The proof of the converse is clear.

(iii) Let $A$ be a summand-ideal of $M$ and $I$ be a normal ideal of $M$. If $x\in I$, then by (ii),
$x=a\oplus b$, for some $a\in A$ and $b\in A^\bot$. Since $b,a\leq x\in I$, then $a,b\in I$, so $a\in I\cap A$ and
$b\in I\cap A^\bot$. It follows that $x\in  (A\cap I)\oplus(I\cap A^\bot)$. Therefore,
$I=(A\cap I)\oplus(I\cap A^\bot)$.

(iv) Let $A\in \mathfrak{Sum}(M)$. Then by (i), $M=A\oplus A^\bot$ and hence $A^\bot\in \mathfrak{Sum}(M)$. Similarly, (i) implies that
$A=A^{\bot\bot}$.

(v) Let $A,B\in \mathfrak{Sum}(M)$. By (iii), $B=(A\cap B)\oplus(A^\bot\cap B)$ and
$B^\bot=(A\cap B^\bot)\oplus(A^\bot\cap B^\bot)$. It follows from (i) that
$M=B\oplus B^\bot=(A\cap B)\oplus(A^\bot\cap B)\oplus (A\cap B^\bot)\oplus(A^\bot\cap B^\bot)=
(A\cap B)\oplus\Big((A^\bot\cap B)\oplus (A\cap B^\bot)\oplus(A^\bot\cap B^\bot)\Big)$.

(1) Since $(A^\bot\cap B)$, $(A\cap B^\bot)$ and $(A^\bot\cap B^\bot)$ are ideals of $M$,
$(A^\bot\cap B)\oplus (A\cap B^\bot)\oplus(A^\bot\cap B^\bot)$ is an ideal of $M$.

(2) We claim $(A\cap B)\cap\Big((A^\bot\cap B)\oplus (A\cap B^\bot)\oplus(A^\bot\cap B^\bot)\Big)=\{0\}$. Indeed,
put $x\in (A\cap B)\cap\Big((A^\bot\cap B)\oplus (A\cap B^\bot)\oplus(A^\bot\cap B^\bot)\Big)$. Similarly to
the proof of part (i), we can see that $x=u\oplus v\oplus w$ for some $u\in A^\bot$, $v\in B^\bot$ and $w\in B^\bot\cap A^\bot$.
So, $u,v,w\in A\cap B$ (since $u,v,w\leq x\in A\cap B$). It follows that $u=v=w=0$ and hence $x=0$.

From (1) and (2) it follows that $A\cap B\in \mathfrak{Sum}(M)$.

(vi) Let $x\in M$. By (i), there are $a\in A$  and $b\in A^\bot$ such that $x=a\oplus b$. Let $x=u\oplus v$ for some $u\in A$ and
$v\in A^\bot$. Then $a\oplus b=u\oplus v$ and so $u=u\wedge (a\oplus b)$. By \cite[Prop. 1.17]{Georgescu}, it follows that
$u\leq (u\wedge a)\oplus (u\wedge b)=u\wedge a$, hence $u\leq a$. A similar argument shows that $a\leq u$, that is $a=u$. In a similar way, we can show that
$b=v$.

(vii) Let $A$ be a summand-ideal of $M$. By (i) and (vi), there are unique elements $a\in A$ and $b\in A^\bot$ such that $1=a\oplus b$.
For each $y\in A$, we have
$y=y\wedge 1=y\wedge (a\oplus b)\leq (y\wedge a)\oplus (y\wedge b)=(y\wedge a)\oplus 0=(y\wedge a)$, hence $y=y\wedge a$ and so
$y\leq a$ which entails $A=\downarrow a$. Since $a\leq a\oplus a\in A$, $a\oplus a \leq a$, we get $a=a\oplus a$. That is $a\in A\cap B(M)$. If $a_1 \in A\cap B(M)$ has the property $A = \downarrow a_1$, then $a=a_1$.
\end{proof}

Note that from Proposition \ref{3.2}(vii) it follows that any summand-ideal $A$ is a {\it Stonean ideal} of $M$, i.e. $A = \downarrow a$ for some Boolean element $a \in A\cap B(M)$.

\begin{cor}\label{3.2.1}
Let $(M;\oplus,^-,^\sim,0,1)$ be a pseudo $MV$-algebra.  Then
\begin{itemize}

\item[{\rm(i)}] A non-empty subset $A$ of $M$ is a summand-ideal of $M$ if and only if $A=\downarrow a$
for some Boolean element $a\in M$.
In such a case, there is a unique $a \in A\cap B(M)$ such that $A = \downarrow a$ and $A^\bot=\downarrow a'$.

\item[{\rm(ii)}]  If $A,B\in \mathfrak{Sum}(M)$, then $A\vee B:=\langle A \cup B\rangle_n \in \mathfrak{Sum}(M)$, and  $\mathfrak{Sum}(M)$ is a Boolean algebra that is isomorphic to $B(M)$.
\end{itemize}
\end{cor}

\begin{proof}

(i) Let $A$ be a summand-ideal of $M$. Then $M=A\boxplus A^\bot$.
By  Proposition \ref{3.2}(vi),
there are unique elements $a\in A$ and $b\in A^\bot$ such that $1=a\oplus b$ and by the proof of
Proposition \ref{3.2}(vii), $a\in B(M)\cap A$ and  $A=\downarrow a$
(similarly, since $A^\bot$ is also a summand-ideal, then $A^\bot=\downarrow b$).
By \cite[Prop. 1.17(1)]{Georgescu}, $a'=a'\wedge 1=a'\wedge (a\oplus b)\leq (a'\wedge a)\oplus (a'\wedge b)=(a'\odot a)\oplus (a'\wedge b)=
a'\wedge b$, hence $a'=a'\wedge b$, that is $a'\leq b$.
It follows that $a'\in A^\bot$. Now, Proposition \ref{3.2}(vi) and $1=a\oplus a'$ imply that
$b=a'$. Therefore, $A^\bot=\downarrow a'$. The uniqueness of $a$ follows from (vii) of Proposition \ref{3.2}.

Conversely, let $A=\downarrow a$ for some Boolean element $a$. By Lemma \ref{le:2}, $A$ is a normal ideal of $M$.
Then clearly, $\downarrow a'$ is an ideal of $M$ (since $a'$ is a Boolean element), $A\cap \downarrow a'=\downarrow a\cap \downarrow a'=\{0\}$
and $\langle A\cup \downarrow a'\rangle_n=M$ and hence, $M=A\boxplus \downarrow a'$.
Therefore, $A$ is a summand-ideal of $M$ and by Proposition \ref{3.2}(i), $A^\bot=\downarrow a'$. In a similar way,
if $A^\bot=\downarrow a'$ for some Boolean element $a\in M$, we can
show that $A=\downarrow a$ is a summand-ideal of $M$.

(ii) Let $A,B \in \mathfrak{Sum}(M)$. By (i), there are unique Boolean elements $a \in A$, $b\in B$ such that $A = \downarrow a$ and $B = \downarrow b$. Then $a\oplus b$ is a Boolean element and $a\oplus b \in \langle A\cup B\rangle_n$. Then $\langle A \cup B\rangle_n = \downarrow (a\oplus b)$ which by (i) says that $A\vee B$ is a summand-ideal of $M$. In a similar way, we can show that $A \wedge B:=A\cap B = \downarrow(a\odot b)$. In addition, $(A\vee B)^\bot = A^\bot\wedge B^\bot$ and $(A\wedge B)^\bot = A^\bot \vee B^\bot$, $\mathfrak{Sum}(M)$ is distributive with respect to $\vee$ and $\wedge$.
Therefore, by (iv) and (v) of Proposition \ref{3.2}, $\mathfrak{Sum}(M)$ is a Boolean algebra.

Finally, the mapping $a\mapsto \downarrow a$, $a\in B(M)$, describes an isomorphism of the Boolean algebras $B(M)$ and $\mathfrak{Sum}(M)$.
\end{proof}

\begin{defn}
A pseudo $MV$-algebra $(M;\oplus,^-,^\sim,0,1)$ is called (i) {\it projectable} if $a^\bot \in \mathfrak{Sum}(M)$ for all $a \in M$, and (ii)  {\it strongly projectable} if $\rho(M)\s \mathfrak{Sum}(M)$.
\end{defn}

As a corollary of \cite[Prop. 6.9]{DvuS}, we have that every projectable pseudo $MV$-algebra is representable.

\begin{cor}\label{pseudo complemented}
Each strongly projectable pseudo $MV$-algebra is a pseudocomplemented lattice.
\end{cor}

\begin{proof}
Let $(M;\oplus,^-,^\sim,0,1)$ be a strongly projectable pseudo $MV$-algebra and $a\in M$.
Then $a^{\bot}$ is a polar ideal of $M$ and so $a^{\bot}$ is a summand-ideal.
By Proposition \ref{3.2}(vii),
$a^{\bot}=\downarrow b$ for some $b\in B(M)$. Clearly, $b$ is a pseudocomplement of $a$, i.e. $x\wedge a=0$ iff $x\le b$.
\end{proof}

We note that according to \cite[Thm 4.2]{151}, every $\sigma$-complete pseudo $MV$-algebra is an $MV$-algebra. The same is true if $M$ is a complete pseudo $MV$-algebra.

\begin{prop}\label{proj-exm}
If $(M;\oplus,^-,^\sim,0,1)$ is a $\sigma$-complete $MV$-algebra, then for each $a\in M$, $M=a^{\bot}\oplus a^{\bot\bot}$.
\end{prop}

\begin{proof}
Let $(M;\oplus,',0)$ be a $\sigma$-complete $MV$-algebra and $a\in M$. Since $a\in a^{\bot\bot}$ and
$a^{\bot\bot}$ is an ideal of $M$, then $n.a\in a^{\bot\bot}$ for all $n\in\mathbb{N}$.
Let $y:=\bigvee\{n.a\mid n\in\mathbb{N}\}$.
From \cite[Prop. 1.18]{Georgescu}, it can be easily obtained that $y\in a^{\bot\bot}$. Also, by
\cite[Prop. 1.21]{Georgescu}, $y\oplus y=y\oplus (\bigvee_{_{n\in\mathbb{N}}}n.a)=
\bigvee_{_{n\in\mathbb{N}}}(y\oplus n.a)=\bigvee_{_{n\in\mathbb{N}}}\bigvee_{_{m\in\mathbb{N}}}
(m.a\oplus n.a)=\bigvee_{_{n\in\mathbb{N}-\{1\}}}n.a=\bigvee_{_{n\in\mathbb{N}}}n.a=y$, so $y\in B(M)$.
We claim that $a^\bot=\downarrow y'$ (clearly, $y'\in B(M)$). If $x\in a^\bot$, then
$x\wedge n.a=0$ for all $n\in\mathbb{N}$ (since $n.a\in a^{\bot\bot}$) and so  \cite[Prop. 1.18]{Georgescu},
$x\wedge y=\bigvee_{n\in\mathbb{N}}(x\wedge n.a)=0$. It follows that
$x=x\wedge (y\vee y')=x\wedge y'$. That is, $x\leq y'$. Now, let $z\in \downarrow y'$.
Then $z\leq y'$ and hence $z\wedge y\leq y'\wedge y=0$. It follows that $z\wedge a=0$ (since $a\leq y$).
Thus $z\in a^\bot$. By Corollary \ref{3.2.1}(ii), $a^{\bot}\in \mathfrak{Sum}(M)$. Therefore,
$M=a^{\bot}\boxplus a^{\bot\bot}$.
\end{proof}

\begin{rmk}
From Proposition \ref{proj-exm}, we know that every $\sigma$-complete pseudo $MV$-algebra is projectable.
Now, assume that $(M;\oplus,^-,^\sim,0,1)$ is a complete pseudo $MV$-algebra. Then $M$ is a complete  $MV$-algebra.
By \cite[Prop 5.8]{Dvu1}, any polar ideal of a complete $MV$-algebra is closed under arbitrary join and so it is
a principal ideal. Hence, by Proposition \ref{proj-exm}, it is a summand ideal of $M$.
That is, any complete $MV$-algebra is strongly projectable.
\end{rmk}

\section{Orthocompletion of pseudo $MV$-algebras}

In the present section, we establish main results of the paper. They will be deal  mainly with representable pseudo $MV$-algebras. Since every $MV$-algebra is a subdirect product of linearly ordered $MV$-algebras, the results are valid also for $MV$-algebras. In such a case, the representing unital $\ell$-group for an $MV$-algebra is of course Abelian.

We recall that two elements $x$ and $y$ of a pseudo $MV$-algebra $M$ are {\it disjoint}  if $x\wedge y = 0$.

\begin{defn}
A pseudo $MV$-algebra $(M;\oplus,^-,^\sim,0,1)$ is called {\it orthocomplete} if
\begin{itemize}
\item[(i)] $M$ is strongly projectable;

\item[(ii)] each set of pairwise disjoint non-zero elements of $M$ has the least upper bound.
\end{itemize}
\end{defn}

A non-empty subset $X$ of non-zero mutually orthogonal elements of a pseudo $MV$-algebra $M$ is said to be {\it disjoint}.

\begin{defn}
Let $(M_1;\oplus,^-,^\sim,0,1)$ be a subalgebra of a pseudo $MV$-algebra $(M_2;\oplus,^-,^\sim,0,1)$. Then $M_1$ is called a {\it large subalgebra} of $M_2$ (or $M_2$ is called an {\it essential extension} of $M_1$) if, for each $y\in M_2-\{0\}$, there are $n\in\mathbb{N}$ and
$x\in M_1-\{0\}$ such that $x\leq n. y$.
\end{defn}

\begin{exm}
Consider the Abelian $\ell$-groups $(\mathbb{Z};+,0)$ and $(\mathbb{Q};+,0)$. Then $M_1:=\Gamma(\mathbb{Z},10)$ is an
$MV$-subalgebra of $M_2:=\Gamma(\mathbb{Q},10)$. It can be easily seen that for each $y\in M_2-\{0\}$, there are
$n\in\mathbb{N}$ and $x\in M_1-\{0\}$ such that $x\leq ny$. Therefore, $M_2$ is an essential extension for $M_1$.
\end{exm}

There is an interesting relation between a pseudo $MV$-algebra and its essential extensions. In the next proposition we establish this relation.
First, we recall the following remark on pseudo $MV$-algebras.

\begin{rmk}\label{ee0}
Let $(M;\oplus,^-,^\sim,0,1)$ be a pseudo $MV$-algebra. We define two ``relative negations" $\ominus_-$ and $\ominus^\sim$ as follows
$$
x\ominus_- y :=x\odot y^-,\quad \mbox{ and } \quad y\ominus^\sim x:= y^\sim \odot x \mbox{ for } x,y \in M.
$$

Then, for all $a,b,c\in M$, we have
\begin{itemize}
\item[(i)] 
$a\ominus_- (b\oplus c) = (a\ominus_- b)\ominus_-c$ and $(a\oplus b)\ominus^\sim c= b\ominus^\sim (a\ominus^\sim c)$ (use \cite[Prop 1.7]{Georgescu}).

\item[(ii)] 
$a\ominus_- b= a^-\ominus^\sim b^-$ and $b \ominus^\sim a=  b^\sim \ominus_- a^\sim$.
\end{itemize}
\end{rmk}

\begin{prop}\label{ee}
Let a pseudo $MV$-algebra $(M_2;\oplus,^-,^\sim,0,1)$ be an essential extension for a pseudo $MV$-algebra $(M_1;\oplus,^-,^\sim,0,1)$. If $S\s M_1$ and $u$ is the least upper bound for $S$ in $M_1$, then
$u$ is the least upper bound for $S$ in $M_2$, too.
\end{prop}

\begin{proof}
Let $S\s M_1$ be given and let $u\in M_1$ be the least upper bound for $S$ in $M_1$. If $S$ is finite, the statement is evident. Thus let $S$ be infinite. Suppose that there exists an upper bound $v\in M_2$ for $S$, and without loss of generality, let us assume that $0< v< u$
(note that, for $u=0$, the proof is clear). Then $0<u\ominus_- v\in M_2$, by the assumption, there are $n\in\mathbb{N}$ and $x\in M_1$ such
that $0< x\leq n.(u\ominus_- v)$ and so $x\ominus_- n.(u\ominus_- v)=0$. For each $s\in S$, $s\leq v$ implies that $u\ominus_- v\leq u\ominus_- s$, whence
for every finite sequence $s_1,s_2,\ldots, s_n$ of elements of $S$, we get
$n.(u\ominus_- v)\leq (u\ominus_- s_1)\oplus\cdots\oplus (u\ominus_- s_n):=\bigoplus_{i=1}^{n}(u\ominus_- s_i)$. 
Hence $0=x\ominus_- n.(u\ominus_- v)\geq x\ominus_- \bigoplus_{i=1}^{n}(u\ominus_- s_i)$, that is
\begin{eqnarray}
\label{R1}  x\ominus_- \bigoplus_{i=1}^{n}(u\ominus_- s_i)=0.
\end{eqnarray}
From Remark \ref{ee0} (i), it follows that $\big(x\ominus_- \bigoplus_{i=2}^{n}(u\ominus_- s_i)\big)\ominus_- (u\ominus_- s_1)=0$, thus
\begin{eqnarray*}
x\ominus_- \bigoplus_{i=2}^{n}(u\ominus_- s_i)&\leq& u\ominus_- s_1=u^{-\sim}\odot  s_1^-\\ 
&\Rightarrow& u^-\oplus \big(x\ominus_- \bigoplus_{i=2}^{n}(u\ominus_- s_i) \big)\leq u^-\oplus (u^{-\sim}\odot s^-_1)\\
&\Rightarrow& u^-\oplus \big(x\ominus_- \bigoplus_{i=2}^{n}(u\ominus_- s_i) \big)\leq s_{1}^-\vee u^-=s_{1}^-,\mbox{ since $s_1\leq u$ }\\
&\Rightarrow& \Big(u^-\oplus \big(x\ominus_- \bigoplus_{i=2}^{n}(u\ominus_- s_i) \big) \Big)^\sim \geq s_1.
\end{eqnarray*}
Since $s_1$ is an arbitrary element of $S$, then
$\Big(u^-\oplus \big(x\ominus_- \bigoplus_{i=2}^{n}(u\ominus_- s_i) \big) \Big)^\sim$ is an upper bound for $S$ which  clearly belongs to
$M_1$. So, by the assumption, $u\leq \Big(u^-\oplus \big(x\ominus_- \bigoplus_{i=2}^{n}(u\ominus_- s_i) \big) \Big)^\sim$. Also,
$u^-\geq u^-\oplus \big( x\ominus_- \bigoplus_{i=2}^{n}(u\ominus_- s_i) \big)\geq u^-$ implies that
$u=\Big(u^-\oplus \big(x\ominus_- \bigoplus_{i=2}^{n}(u\ominus_- s_i) \big)\Big)^\sim$ and
$u^-=u^-\oplus \big(x\ominus_- \bigoplus_{i=2}^{n}(u\ominus_- s_i) \big) $. Hence,
$$
0=u\odot u^-=u\odot \Big(u^-\oplus \big(x\ominus_- \bigoplus_{i=2}^{n}(u\ominus_- s_i) \big)\Big)=\Big(x\ominus_- \bigoplus_{i=2}^{n}(u\ominus_- s_i)\Big)\wedge u.
$$
Since $x\leq n.(u\ominus_- v)\leq \bigoplus_{i=1}^{n}(u\ominus_- s_i)$, then by Remark \ref{ee0}(i),
$0=x\ominus_-\big(\bigoplus_{i=1}^{n}(u\ominus_- s_i)\big)=\big(x\ominus_-(\bigoplus_{i=2}^{n}(u\ominus_- s_i))\big)\ominus_- (u\ominus_- s_1)$ and so
$x\ominus_- \bigoplus_{i=2}^{n}(u\ominus_- s_i)\leq u\ominus_- s_1\leq u$.
It follows that $0=\big(x\ominus_- \bigoplus_{i=2}^{n}(u\ominus_- s_i) \big)\wedge u=
x\ominus_- \bigoplus_{i=2}^{n}(u\ominus_- s_i)$. Now, we return to (\ref{R1}), repeating this process, it can be easily shown that
$x=0$, which is a contradiction. Therefore, $u$ is the least upper bound for $S$ in $M_2$.
\end{proof}

\begin{lem}\label{3.6}
Let $(M_2;\oplus,^-,^\sim,0,1)$ and $(M_1;\oplus,^-,^\sim,0,1)$ be strongly projectable pseudo $MV$-algebras such that $M_1$ is a subalgebra of $M_2$.
Then, for each ideal $I\in \rho(M_1)$, there is a unique Boolean element $b\in M_1$ such that
$I=\downarrow_{_{M_1}} b$ and $(I^{\bot_{M_1}})^{\bot_{M_2}}=\downarrow_{_{M_2}} b$.
\end{lem}

\begin{proof}
Let $I$ be a polar ideal of $M_1$; then $I$ is normal. By Proposition \ref{3.2}(vi) and Corollary \ref{3.2.1}(ii),
there is $b\in M_1$ such that $b\oplus b=b$, $I=\downarrow_{_{M_1}} b$,
$I^{\bot_{M_1}}=\downarrow_{_{M_1}} b'$, and $b\oplus b'=1$ is the only decomposition of $1$ in $M_1=I\boxplus  I^{\bot_{M_1}}$.
Clearly, $(I^{\bot_{M_1}})^{\bot_{M_2}}\in \rho(M_2)$ and $b\in I\s (I^{\bot_{M_1}})^{\bot_{M_2}}$ (since $I$ is a polar ideal of $M_1$),
so by
Proposition \ref{3.2}(vi), there are unique elements $u\in (I^{\bot_{M_1}})^{\bot_{M_2}}$ and $v\in ((I^{\bot_{M_1}})^{\bot_{M_2}})^{\bot_{M_2}}$
such that $b'=u\oplus v$. Since $u\in (I^{\bot_{M_1}})^{\bot_{M_2}}$, and $b'\in I^{\bot_{M_1}}$ then $u\wedge b'=0$ and so
$u=0$ (since $b'=u\oplus v$ implies that $u\leq b'$). That is, $b'=v\in ((I^{\bot_{M_1}})^{\bot_{M_2}})^{\bot_{M_2}}$ and hence
by $b\oplus b'=1$ and Proposition \ref{3.2}(vi),(vii) and Corollary \ref{3.2.1},
$\downarrow_{_{M_2}} b=(I^{\bot_{M_1}})^{\bot_{M_2}}$ and $\downarrow_{_{M_2}} b'=((I^{\bot_{M_1}})^{\bot_{M_2}})^{\bot_{M_2}}$.
\end{proof}

Similarly to the proof of \cite[Thm. 8.1.1]{Anderson}, we can show the following lemma.
In fact, the proof of \cite[Thm. 8.1.1]{Anderson} works also for pseudo $MV$-algebras.

\begin{lem}\label{3.7}
If $(M_2;\oplus,^-,^\sim,0,1)$ is an essential extension for a pseudo $MV$-algebra $(M_1;\oplus,^-,^\sim,0,1)$, then
$\rho(M_1)$ and $\rho(M_2)$ are lattice isomorphic under the map $\Phi: \rho(M_2)\ra \rho(M_1)$ and
$\Psi:\rho(M_1)\ra \rho(M_2)$ defined by $\Phi(I)=I\cap M_1$ and $\Psi(J)=(J^{\bot_{M_1}})^{\bot_{M_2}}$.
\end{lem}

\begin{cor}\label{3.8}
Let $(M_2;\oplus,^-,^\sim,0,1)$ and $(M_1;\oplus,^-,^\sim,0,1)$ be strongly projectable pseudo $MV$-algebras such that $M_1$ is a large subalgebra of $M_2$.
\begin{itemize}
\item[{\rm(i)}] For each polar ideal $I$ of $M_2$, there is a unique Boolean element $a\in M_1\cap I$ such that $I=\downarrow_{M_2} a$.

\item[{\rm(ii)}] Let $I\in \rho(M_1)$. Then there is $a\in B(M_1)$ such that $I=\downarrow_{_{M_1}}a$ if and only if $(I^{\bot_{M_1}})^{\bot_{M_2}}=\downarrow_{_{M_2}}a$.
\end{itemize}
\end{cor}

\begin{proof}
(i) Let $J$ be a polar ideal of $M_2$. By Lemma \ref{3.7}, $J=((J\cap M_1)^{\bot_{M_1}})^{\bot_{M_2}}$ and $J\cap M_1$ is a polar ideal of $M_1$.
By Lemma \ref{3.6}, there is a unique element $a\in B(M_2)$ such that $J\cap M_1=\downarrow_{_{M_1}}a$ and
$(J\cap M_1)^{\bot_{M_1}})^{\bot_{M_2}}=\downarrow_{_{M_2}}a$. Therefore, $J=\downarrow_{_{M_2}}a$.

(ii) By Lemma \ref{3.6}, if $I=\downarrow_{_{M_1}}a$, then
$(I^{\bot_{M_1}})^{\bot_{M_2}}=\downarrow_{_{M_2}}a$. Conversely, if
$(I^{\bot_{M_1}})^{\bot_{M_2}}=\downarrow_{_{M_2}}a$, then by (i),
$a\in M_1\cap \downarrow_{_{M_2}}a$ and by Lemma \ref{3.7},
$\downarrow_{_{M_1}}a=M_1\cap \downarrow_{_{M_2}}a=M_1\cap (I^{\bot_{M_1}})^{\bot_{M_2}}=I$.
\end{proof}

\begin{thm}\label{3.9}
If $(A;\oplus,^-,^\sim,0,1)$ is a large pseudo $MV$-subalgebra of a strongly projectable pseudo $MV$-algebra $(B;\oplus,^-,^\sim,0,1)$ and $\{M_i\}_{i\in J}$ is the set of all
strongly projectable pseudo $MV$-subalgebras of $B$ containing $A$, then $M:=\bigcap_{i\in J}M_i$ is a strongly projectable pseudo $MV$-algebra containing $A$.
\end{thm}

\begin{proof}
Let $I$ be a polar ideal of $M$. By Corollary \ref{3.2.1}, it suffices to show that there is a Boolean element $a\in B$ such that
$I=\downarrow_{M} a$. By Lemma \ref{3.7}, for each $i\in J$, $(I^{\bot_{M}})^{\bot_{M_i}}$ is a
polar ideal of $M_i$ and so by Corollary \ref{3.8}(ii),
there is a unique Boolean element $a_i\in M_i$ such that $(((I^{\bot_{M}})^{\bot_{M_i}})^{\bot_{M_i}})^{\bot_{B}}=\downarrow_{_{B}}a_i$ and
$(I^{\bot_{M}})^{\bot_{M_i}}=\downarrow_{_{M_i}}a_i$. Also,
$(((I^{\bot_{M}})^{\bot_{M_i}})^{\bot_{M_i}})^{\bot_{B}}\cap M=(((I^{\bot_{M}})^{\bot_{M_i}})^{\bot_{M_i}})^{\bot_{B}}\cap M_i\cap M=
(I^{\bot_{M}})^{\bot_{M_i}}\cap M=I$ and $(I^{\bot_{M}})^{\bot_{B}}\cap M=I$ and so by Lemma \ref{3.7},
$(((I^{\bot_{M}})^{\bot_{M_i}})^{\bot_{M_i}})^{\bot_{B}}=(I^{\bot_{M}})^{\bot_{B}}$ (since $\Phi:\rho(B)\ra \rho(M)$ is one-to-one).
It follows that $a_i=a_j=:a\in M$ for all $i,j\in J$, that is there exists a unique Boolean element $a\in M$ such that
$(I^{\bot_{M}})^{\bot_{M_i}}=\downarrow_{_{M_i}}a$. Hence,
$\bigcap_{i\in J} \downarrow_{_{M_i}}a = \bigcap_{i\in J} (I^{\bot_{M}})^{\bot_{M_i}}$.
Also, $\bigcap_{i\in J} \downarrow_{_{M_i}}a=\{x\in B\mid x\leq a,\ x\in M_i,\ \forall i\in J\}=\downarrow_{_{M}}a$ and
$\bigcap_{i\in J}(I^{\bot_{M}})^{\bot_{M_i}}=\bigcap_{i\in J}(I^{\bot_{M}})^{\bot_{M_i}}\cap M=I$,
so $I=\downarrow_{_{M}}a$ which proves that $M$ is a strongly projectable pseudo $MV$-algebra.
\end{proof}

\begin{rmk}\label{3.9.1}
Similarly to the proof of Theorem \ref{3.9}, we can show that if
$(A;\oplus,^-,^\sim,0,1)$ is a large subalgebra of a pseudo $MV$-algebra
$(M;\oplus,^-,^\sim,0,1)$ and
$B$ and $C$ are strongly projectable pseudo $MV$-subalgebra of $M$ containing $A$, then $B\cap C$ is also a strongly projectable pseudo $MV$-algebra.
\end{rmk}

\begin{defn}
A minimal orthocomplete pseudo $MV$-algebra containing $M$ as a large pseudo $MV$-subalgebra is called an {\it orthocompletion} for $M$.
\end{defn}

\begin{rmk}\label{rmk}
Let $G$ be a representable $\ell$-group. We recall that $O(G)$ is an orthocomplete $\ell$-group constructed by the following process
(for more details, we refer to \cite{Anderson,Conrad1,Darnel}).
Let $\phi: G\hookrightarrow \prod_{\lambda\in \Lambda} G_\lambda$ be a subdirect embedding,
where $G_\lambda$ is a totally ordered $\ell$-group for all $\lambda\in\Lambda$.
Suppose that $B(\Lambda)=\{Supp(I)\mid I\in \rho(G)\}$, where $Supp(X)=\bigcup\{Supp(x)\mid x\in X\}$ and
$Supp(x)=\{\lambda\in\Lambda \mid x(\lambda)\neq 0\}$ for
each subset $X$ of $G$ and each $x\in G$. Let $\{f_\alpha\}_{\alpha\in \Omega}\s G$, $D(G)$ be the set of all maximal pairwise disjoint
subsets of $B(\Lambda)$ and $\{F_\alpha\}_{\alpha\in \Omega}\in D(G)$, we say that $\{f_\alpha,F_\alpha\}_{\alpha\in \Omega}$ {\it underlines}  an
element $x\in G$ if, for each $\alpha\in\Omega$ and each $\lambda\in F_\alpha$, $x(\lambda)=f_\alpha(\lambda)$.
Set $L=\{f\in \prod_{\lambda\in \Lambda} G_\lambda\mid \mbox{there exists $\{g_\alpha,G_\alpha\}_{\alpha\in \Omega}$ underlying $f$ }\}$.
Then $L$ is an $\ell$-subgroup of $\prod_{\lambda\in \Lambda} G_\lambda$ and the relation $\theta$, which is defined by
$(x,y)\in \theta$ if and only if
there exists $\{f_\alpha,F_\alpha\}_{\alpha\in \Sigma}$ underlines both
$x$ and $y$, is a congruence relation on the
$\ell$-group $L$, so $L/\theta$ (the set of all equivalence classes of $L$ under $\theta$) is an $\ell$-group and $O(G):=L/\theta$.
Let $\pi:L\ra O(G)$ be the natural projection map. Then $Im(\phi)\s L$ and $\xi_{_G}:=\pi\circ \phi:G\ra O(G)$ is an injective $\ell$-group homomorphism.
In fact, $O(G)$ is an orthocompletion for $\xi_{_G}(G)$ (note that $\xi_{_G}(G)\cong G$). From now on, in this paper, we suppose
$G$ is an $\ell$-subgroup of $O(G)$.
\end{rmk}

\begin{thm}\label{ortho}
Each representable pseudo $MV$-algebra has an orthocompletion. Moreover, any two such orthocompletions are isomorphic.
\end{thm}

\begin{proof}
Let $(A;\oplus,^-,^\sim,0,1)$ be a representable pseudo $MV$-algebra. By Theorem \ref{functor}, there exists a representable
$\ell$-group $(G_A;+,0)$ with strong unit $u_A$ such that
$A\cong \Gamma(G_A,u_A)$. Since $G_A$ is representable
(see \cite[Cor. 4.1.2]{Anderson}), by \cite[Thm. 8.1.3]{Anderson} or \cite[Thm 48.2]{Darnel}, it has a unique orthocompletion
which is denoted by $O(G_A)$.
Since $u_A$ is a strong unit of $G_A$ and $G_A\leq O(G_A)$, then $u_A$ is a positive element of $O(G_A)$
and so $B:=\Gamma(O(G_A),u_A)$ is a pseudo $MV$-algebra (see \cite[Prop. 2.1.2]{mv-book}) and clearly, $\Gamma(G_A,u_A)$ is a pseudo $MV$-subalgebra of $B$.
It follows that $A$ is isomorphic to a pseudo $MV$-subalgebra of $B$. We claim that $B$ is an orthocomplete pseudo $MV$-algebra.

(1) We assert $A$ is a large pseudo $MV$-subalgebra of $B$. First, using mathematical induction, we have if $a_1,\ldots, a_n \in A$, then $(a_1\oplus \cdots \oplus a_n)=(a_1+\cdots+a_n)\wedge u_A$. If $n=1,2$, the statements is clear. Using distributivity of the group addition $+$ with respect to $\wedge$ in the $\ell$-group, we have $(a_1\oplus a_2)\oplus a_3 = \big(((a_1+a_2)\wedge u_A)+a_3\big)\wedge u_A=(a_1+a_2+a_3)\wedge (u_A+a_3)\wedge u_A=(a_1+a_2+a_3)\wedge u_A$.

Put $b\in B$. Since $O(G_A)$ is an orthocompletion of $G_A$, $0< b\leq u_A$ and $b\in O(G_A)$,
then there exist $n\in\mathbb{N}$ and a strictly positive element $x\in G_{A}$ such that $x \leq nb$. From $0<b$, it follows $x\wedge u_A \leq (nb)\wedge u_A= n.b$,
hence $x\wedge u_A\in \Gamma(G_A,u_A)$ and $x\wedge u_A\leq n.b$, and finally  $A$ is a large pseudo $MV$-subalgebra of $B$.

(2) Let $S$ be a pairwise disjoint subset of $B$. Then clearly, $S$ is a pairwise disjoint subset of $O(G_A)$, so by the assumption,
$\bigvee S\in O(G_A)$. Since $u_A$ is an upper bound for $S$ in $O(G_A)$, $\bigvee S\in B$.

(3) Let $I$ be a polar ideal of $B$. Then there exists a subset $X$ of $B$ such that $I=X^\bot=\{b\in B\mid b\wedge x=0, \forall x\in X\}$.
Set $\overline{I}=\{g\in O(G_A)\mid |g|\wedge x=0, \forall x\in X\}$. It is easy to show that $I=\overline{I}\cap B$ and it is a polar $\ell$-subgroup of $O(G_A)$ and so there exists an $\ell$-subgroup
$J$ of $O(G_A)$ such that $O(G_A)=\overline{I}+ J$. Let $K=J\cap B$. By Theorem \ref{ideals}, $K$ is a normal ideal of $B$. Clearly, $K\cap I=\{0\}$. Let
$b\in B$. Then $0\leq b$ and there exist $b_1\in \overline{I}$ and $b_2\in J$ such that $b=b_1+b_2$.
By \cite[Prop. 1.1.3a]{Anderson} or $(\ell$3), we have $b=|b|=|b_1+b_2|\leq |b_1|+|b_2|+|b_1|$. Since in any $\ell$-group, \cite[Prop. 1.1.5]{Anderson}, for all positive elements $g,h,u$,  we have $(g+h)\wedge u\le (g\wedge u)+ (h\wedge u)$, we get 
$b=b\wedge u_A\leq (|b_1|+|b_2|+|b_1|)\wedge u_A\leq (|b_1|\wedge u_A)\oplus (|b_2|\wedge u_A)\oplus (|b_1| \wedge u_A)$. Clearly,
$|b_1|\wedge u_A\in \overline{I}\cap B=I$ and $|b_2|\wedge u_A\in J\cap B=K$ so
$(|b_1|\wedge u_A)\oplus (|b_2|\wedge u_A)\oplus (|b_1| \wedge u_A)\in \langle I\cup J\rangle_n$. Hence $x\in \langle I\cup J\rangle_n$, whence
$B=I\boxplus J$. That is, $I\in \mathfrak{Sum}(B)$.

From (1), (2) and (3) it follows that $B$ is an orthocomplete pseudo $MV$-algebra. Now, we show that it is an orthocompletion for $A$.
Let $M$ be an orthocomplete pseudo $MV$-algebra such that $\Gamma(G_A,u_A)$ is a subalgebra of  $M$ and $M$ is a subalgebra of $\Gamma(O(G_A),u_A)=B$. Put $x\in B$. Then $x=[b]$ for some
$b\in L$, where $[b]$ is the congruence class of $b$ in $L$, hence by the proof of \cite[Thm 48.2, p. 313]{Darnel}, there is $\{f_\alpha,F_\alpha\}_{\alpha\in \Omega}$ underlying $b$ and
$\bigvee_{\alpha\in\Omega}[f_\alpha]$ exists and is equal to $[b]$ (we recall that in the proof of the mentioned theorem, it was proved that
$\{[f_\alpha]\}_{\alpha\in\Omega}$ is a pairwise disjoint subset of positive elements of $G_A$ that $\bigvee_{\alpha\in\Omega}[f_\alpha]=[b]\leq u_A$.
Hence  $\{[f_\alpha]\}_{\alpha\in\Omega}\s \Gamma(G_A,u_A)$). Since $\Gamma(G_A,u_A)\s M$ and $M$ is orthocomplete, $[b]\in M$ and so $M=B$. Therefore,
$B$ is an orthocompletion for $\Gamma(G_A,u_A)$. Finally, we will show that if $B_1$ is another orthocompletion for $A$, then $B\cong B_1$.
Let $B_1$ be an orthocompletion for the pseudo $MV$-algebra $A$. Then there is an injective $MV$-homomorphism $i:A\ra B_1$. We know that $\pi\circ\phi:G_A\ra
O(G_A)$ is an injective $\ell$-group homomorphism and $\pi\circ\phi:\Gamma(G_A,u_A)\ra \Gamma(O(G_A),u_A)$ is an injective pseudo $MV$-homomorphism (see the notations in Remark \ref{rmk}).
Let $\alpha:A\ra \Gamma(G_A,u_A)$ be an isomorphism of pseudo $MV$-algebras. Then $\pi\circ\phi\circ \alpha:A\ra B$ is a one-to-one pseudo $MV$-homomorphism. Since $B_1$ is
orthocomplete, then by the above results, $B_1=\Gamma(O(G_{B_1}),u_{B_1})$ (up to isomorphic image). From $A\cong i(A)\leq B_1$ it follows that
$G_A\cong G_{i(A)}\leq G_{B_1}$ and $O(G_{i(A)})\leq O(G_{B_1})$, hence
$B=\Gamma(O(G_{A}),u_{A})\cong\Gamma(O(G_{i(A)}),u_{i(A)})\leq \Gamma(O(G_{B_1}),u_{B_1})=B_1$. Moreover, $\Gamma(O(G_{i(A)}),u_{i(A)})$ is an orthocompletion for $i(A)$.
By summing up the above results, we get that $i(A)\leq \Gamma(O(G_{i(A)}),u_{i(A)})\leq B_1$. Since $B_1$ is an orthocompletion of $i(A)$,
$\Gamma(O(G_{i(A)}),u_{i(A)})= B_1$ and so $B\cong B_1$.
We must note that,  since $i(A)\leq B_1$,
then from the proof of \cite[Prop. 2.4.4]{mv-book}, we get $u_{i(A)}=u_{B_1}$.
\end{proof}

In Theorem \ref{ortho}, we used an orthocompletion of a representable $\ell$-group to construct an orthocompletion of a representable pseudo $MV$-algebra. In the next theorem,
we will show that if $(A;\oplus,^-,^\sim,0,1)$ is an orthocomplete representable pseudo $MV$-algebra such that $u_A$ is a strong unit of the $\ell$-group
$G_A$, then $G_A$ is also an orthocomplete $\ell$-group.

\begin{thm}\label{3.3}
Let $(A;\oplus,^-,^\sim,0,1)$ be an orthocomplete representable pseudo $MV$-algebra such that $u_A$ is a strong unit of the representable $\ell$-group
$O(G_A)$. Then $G_A$ is an orthocomplete $\ell$-group.
\end{thm}

\begin{proof}
By Theorem \ref{functor}, we know that $A\cong \Gamma\Xi(A)=\Gamma(G_A,u_A)$. Also, $\Gamma(G_A,u_A)$ is a pseudo $MV$-subalgebra of
$\Gamma(O(G_A),u_A)$. Set $B=\Gamma(O(G_A),u_A)$. Then there is a one-to-one homomorphism of  pseudo $MV$-algebras $f:A\ra B$. Since $A$ and $B$ are orthocomplete
and $f(A)\leq B$, then by Theorem \ref{ortho}, $f(A)=B$ and so $A\cong B$. Hence by Theorem \ref{functor}, $\Xi(f): \Xi(A)\ra \Xi(B)$ is an
isomorphism. It follows that $(G_A,u_A)\cong\Xi(\Gamma(G_A,u_A))\cong\Xi(A)\cong \Xi(B)\cong \Xi(\Gamma(O(G_A),u_A))\cong (O(G_A),u_A)$
(note that, since $u_A$ is a strong unit of $O(G_A)$, then $\Xi(\Gamma(O(G_A),u_A))\cong (O(G_A),u_A)$).
Therefore, $G_A$ is orthocomplete.
\end{proof}

In Corollary \ref{3.5}, we try to find a representable pseudo $MV$-algebra $(A;\oplus,^-,^\sim,0,1)$ such that $u_A$ is a strong unit for $O(G_A)$.

A pseudo $MV$-algebra $(A;\oplus,^-,^\sim,0,1)$ is called {\it finite representable} if there exists a subdirect embedding $\alpha$ from $A$ into a finite direct product
of pseudo $MV$-chains. It is easy to see that $A$ is finite representable if there is a finite subset $S$ of prime and normal ideals of $A$ such that $\bigcap S=\{0\}$.
Similarly, we can define a finite representable $\ell$-group.

\begin{rmk}\label{3.4}
Let $(G;+,0)$ be a finite representable $\ell$-group with strong unit $u$.
Then there is a subset $\{P_1,\ldots,P_n\}$ of prime $\ell$-ideals of $G$ such that $\bigcap_{i=1}^{n}P_i=\{0\}$.
Clearly, the natural map $\varphi:G\ra \prod_{i=1}^{n}G/P_i$ sending $g$ to $\varphi(g)=(g/P_1,\ldots,g/P_n)$
is a subdirect embedding of $\ell$-groups.
We claim that $u$ is a strong unit of $O(G)$, where $O(G)$ is an orthocompletion of $G$.
Put $x\in O(G)$. Then by \cite[Thm. 8.1.3]{Anderson}, $x=[(x_1/P_1,\ldots,x_n/P_n)]$
for some $(x_1/P_1,\ldots,x_n/P_n)\in L$.
Since $(u/P_1,\ldots,u/P_n),(x_1/P_1,\ldots,x_n/P_n)\in L$ (see the notations in Remark \ref{rmk}), there are
$\{f_\alpha,F_\alpha\}_{\alpha\in A}$ and $\{g_\beta,G_\beta\}_{\beta\in B}$ that underline $(u/P_1,\ldots,u/P_n)$ and $(x_1/P_1,\ldots,x_n/P_n)$,
respectively. It follows that
$$\forall \alpha\in A,\   \forall \lambda\in F_\alpha, \ f_\alpha(\lambda)=u/P_\lambda \quad
\forall \beta\in B,\   \forall \lambda\in G_\beta, \ g_\beta(\lambda)=x_\lambda/P_\lambda.$$
Since $u$ is a strong unit of $G$, there is $m\in \mathbb{N}$ such that $x_i\leq mu$ for all $i\in\{1,2,\ldots,n\}$.
Clearly, $\{mf_\lambda,F_\lambda\}$ underlines $(mu/P_1,\ldots, mu/P_n)$ and
for all $\alpha\in A$ and $\beta\in B$ and $\lambda\in F_\alpha\cap G_\beta$, we have
$g_{\beta}(\lambda)=x_\lambda/P_\lambda\leq mu/P_\lambda=mf_\alpha(\lambda)$, which implies that $(u/P_1,\ldots,u/P_n)$
is a strong unit of $O(G)$.
\end{rmk}

\begin{cor}\label{3.5}
Let $(M;\oplus,^-,^\sim,0,1)$ be a finite representable pseudo $MV$-algebra. Then $u_M$ is a strong unit of $\Xi(M)=(G_M,u_M)$ and it
is an orthocomplete representable  $\ell$-group.
\end{cor}

\begin{proof}
Since $M$ is a finite representable pseudo $MV$-algebra, by Theorem \ref{ideals}, it is clear that $G_A$ is a finite representable
$\ell$-group, hence by Remark \ref{3.4}, $(O(G_A),u_A)$ is a unital $\ell$-group. So by Theorem \ref{3.3}, $\Xi(M)=(G_M,u_M)$ is an orthocomplete  $\ell$-group. Therefore, by {\rm \cite[Prop. 48.1]{Darnel}}, it is representable.
\end{proof}

\begin{lem}\label{lem1}
Let $G$ be an $\ell$-subgroup of an $\ell$-group $H$ and $u\in G$ be a strong unit of $H$. If the pseudo $MV$-algebra $\Gamma (G,u)$ is a large pseudo
$MV$-subalgebra of $\Gamma(H,u)$, then $G$ is a large $\ell$-subgroup of $H$.
\end{lem}

\begin{proof}
Put $0<h\in H$. Then $u\wedge h\in \Gamma(H,u)$. If $u\wedge h=0$, then (since $u$ is a strong unit) there exists $n\in\mathbb{N}$ such that
$h \leq nu$ and so $h = h\wedge (nu)\leq n(h\wedge u)=0$ which is a contradiction and so $h\wedge u\neq 0$.
%
%
By the assumption, there are $m\in\mathbb{N}$ and $x\in \Gamma(G,u)-\{0\}$ such that $x \leq m.h \leq mh$. 
Therefore, $G$ is a large $\ell$-subgroup of $H$ (equivalently, $H$ is an essential extension for $G$).
\end{proof}

\begin{thm}\label{lem2}
If a pseudo $MV$-algebra $(M;\oplus,^-,^\sim,0,1)$ is an essential extension for a pseudo $MV$-algebra $A$, then the unital $\ell$-group $(G_M,u_M)$ is an essential extension for the $\ell$-group $(G_A,u_A)$.
\end{thm}

\begin{proof}
Let $M$ be an essential extension for the pseudo $MV$-algebra $A$. By Theorem \ref{functor}, we have
$A\cong \Gamma(G_A,u_A)$ and $M\cong\Gamma(G_M,u_M)$ and the following diagram are commutative.
$$\begin{array}{ll}
\begin{tabular}{ccc}
A& $\xrightarrow{ \quad  f \quad  }$ & M\\
$\downarrow$ & & $\downarrow$ \\
$\Gamma(\Xi (A))$& $\xrightarrow{\Gamma(\Xi (f))}$ & $\Gamma(\Xi (M))$
\end{tabular}
\end{array}$$
It follows that $\Gamma(G_M,u_M)=\Gamma(\Xi (M))$ is an essential extension for
$\Gamma(G_A,u_A)=\Gamma(\Xi (A))$ and hence by Lemma \ref{lem1}, $G_M$ is an essential extension for
the $\ell$-group $G_A$.
\end{proof}

\begin{rmk}\label{3.10}
Let a representable $\ell$-group $(H;+,0)$ be an essential extension for an $\ell$-group $G$, then also $G$ is representable because representable $\ell$-groups form a variety, \cite[p. 304]{Darnel},
and let $\mathfrak{D}(G)$
and $\mathfrak{D}(H)$ be the set of maximal disjoint subsets of
$\rho(G)$ and $\rho(H)$, respectively. By \cite[Thm 8.1.1]{Anderson}, these
lattices are isomorphic, under the maps $\Phi:\rho(H)\ra \rho(G)$
and $\Psi:\rho(G)\ra \rho(H)$, define by $\Phi(I)=I\cap G$ and
$\Psi(J)=(J^{\bot_G})^{\bot_H}$ for all $I\in \rho(H)$ and $J\in \rho(G)$.
It can be easily seen that $\Phi$ and $\Psi$ can be extend to isomorphisms
between $\mathfrak{D}(G)$ and $\mathfrak{D}(H)$. In fact,
$\mathfrak{D}(G)=\{\Phi(S)\mid S\in \mathfrak{D}(H)\}$. Put $S\in \mathfrak{D}(H)$.
For each $I\in S$, define $\mu_{_I}: G/(I\cap G)^{\bot_G}\ra H/I^{\bot_H}$, by
$\mu_{_I}(x+(I\cap G)^{\bot_G})=x+I^{\bot_H}$.

(1) If $x,y\in G$ such that $x+(I\cap G)^{\bot_G}=y+(I\cap G)^{\bot_G}$, then $x-y\in (I\cap G)^{\bot_G}$, so
$$I=\Psi\circ\Phi(I)=(I\cap G)^{\bot_G\bot_H}\s (x-y)^{\bot_H}\Rightarrow x-y\in (x-y)^{\bot_H\bot_H}\s I^{\bot_H}.$$
It follows that $x+I^{\bot_H}=y+I^{\bot_H}$.

(2) Clearly, $\mu_{_I}$ is an $\ell$-group homomorphism. Moreover, $\mu_{_I}(x+(I\cap G)^{\bot_G})=0+I^{\bot_H}$ implies that
$$x\in  I^{\bot_H} \Rightarrow I\s x^{\bot_H}\Rightarrow I\cap G\s x^{\bot_H}\cap G=x^{\bot_G}\Rightarrow x\in (I\cap G)^{\bot_G}$$
so $\mu_{_I}$ is a one-to-one $\ell$-group homomorphism.

Define $\mu_{_S}:\prod_{I\in S}G/(I\cap G)^{\bot_G}\ra\prod_{I\in S} H/I^{\bot_H}$, by
$\mu_{_S}\big( (x_{_I}+(I\cap G)^{\bot_G})_{_{I\in S}}\big)=(x_{_I}+I^{\bot_H})_{_{I\in S}}$.
From (1) and (2), we get that $\mu_{_S}$ is a one-to-one $\ell$-group homomorphism.
For each $S\in \mathfrak{D}(H)$, set
$G_S=\prod_{I\in S}G/(I\cap G)^{\bot_G}$ and $H_S=\prod_{I\in S}H/I^{\bot_H}$.
Now, let $S,T\in\mathfrak{D}(H)$ such that $S\leq T$
(that is, each $I\in S$, is contained in some $J\in T$). Then the natural map $\pi_{T,S}:G_T\ra G_S$ is an $\ell$-group homomorphism and
by \cite[Thm. 2.6]{Conrad1}, $O(G)$ ($O(H)$) is a direct limit of the family $\{G_S,\pi^G_{_{T,S}}\}_{S\leq T\in \mathfrak{D}(H)}$
($\{H_S,\pi^{H}_{_{T,S}}\}_{S\leq T\in \mathfrak{D}(H)}$), $O(G)$ $(O(H))$
is the orthocompletion of $G$ ($H$), and
$\mu_{_S}:\{G_S,\pi^{G}_{_{T,S}}\}_{S\leq T\in \mathfrak{D}(H)}\ra
\{H_S,\pi^{H}_{_{T,S}}\}_{S\leq T\in \mathfrak{D}(H)}$
is a morphism between these directed systems.
Similarly to the first step of the proof of \cite[Thm. 3.5]{Conrad2}, there is a one-to-one $\ell$-group homomorphism $\mu$
induced by $\{\mu_{_S}\}_{S\in \mathfrak{D}(H)}$ such that the following  diagram is commutative:
\begin{eqnarray}\label{dig 1}
\begin{array}{ll}
\begin{tabular}{ccc}
$G$& $\xrightarrow{ \quad  \s \quad  }$ & $H$\\
$\downarrow{\alpha}$ & & $\downarrow{\beta}$ \\
$O(G)$& $\xrightarrow{\quad \mu \quad }$ & $O(H)$
\end{tabular}
\end{array}
\end{eqnarray}
where $\alpha$ and $\beta$ are the natural one-to-one $\ell$-group homomorphisms introduced in  \cite[Thm. 3.5]{Conrad2}. Moreover,
$O(G)$ and $O(H)$ are orthocompletions of $Im(\alpha)$ and $Im(\beta)$, respectively.
\end{rmk}

In the next theorem, we use an orthocompletion for a representable pseudo $MV$-algebra to show that, for each representable pseudo
$MV$-algebra $M$, a minimal strongly projectable essential extension for $M$ exists.

\begin{thm}\label{3.11}
Let $(M;\oplus,^-,^\sim,0,1)$ be a minimal strongly projectable essential extension for a representable pseudo $MV$-algebra $A$ and $B$ be an orthocompletion for the pseudo $MV$-algebra $A$. If $D$ is the intersection of all projectable pseudo $MV$-subalgebras of $B$ containing $A$, then $M\cong D$.
\end{thm}

\begin{proof}
Let $i:A\ra M$ be the inclusion map. Then by Theorem \ref{functor}, $\Xi(i):(G_A,u_A)\ra (G_M,u_M)$ is an injective homomorphism of unital $\ell$-groups.
Since $A\cong\Gamma(G_A,u_A)$ and $M\cong\Gamma(G_M,u_M)$, then $\Gamma(G_M,u_M)$ is an essential extension
for the pseudo $MV$-algebra $\Gamma(G_A,u_A)$ and so by Lemma \ref{lem1}, $G_M$ is an essential extension  for $G_A$. By Remark \ref{3.10}, we have the following commutative diagram:
\begin{eqnarray}\label{dig 12}
\begin{array}{ll}
\begin{tabular}{ccc}
$G_A$& $\xrightarrow{ \quad \Xi(i) }$ & $G_M$\\
$\downarrow{\alpha}$ & & $\downarrow{\beta}$ \\
$O(G_A)$& $\xrightarrow{\quad \mu \quad }$ & $O(G_M)$
\end{tabular}
\end{array}
\end{eqnarray}
Now, we apply the functor $\Gamma$ and we get the commutative diagram
\begin{eqnarray}\label{dig 3}
\begin{array}{ll}
\begin{tabular}{ccc}
$\Gamma(G_A,u_A)$& $\xrightarrow{ \quad  \Gamma(\Xi(i)) \quad  }$ & $\Gamma(G_M,u_M)$\\
$\downarrow{\Gamma(\alpha)}$ & & $\downarrow{\Gamma(\beta)}$ \\
$\Gamma(O(G_A),u_A)$& $\xrightarrow{\quad \Gamma(\mu) \quad }$ & $\Gamma(O(G_M),u_M)$
\end{tabular}
\end{array}
\end{eqnarray}
Hence, Theorem \ref{functor} implies that there are one-to-one pseudo $MV$-homomorphisms $f:A\ra \Gamma(O(G_A),u_A)$ and
$g: M\ra \Gamma(O(H_A),u_M)$ such that the following diagram commutes.
\begin{eqnarray}\label{dig 4}
\begin{array}{ll}
\begin{tabular}{ccc}
$A$& $\xrightarrow{ \quad  \s \quad  }$ & $M$\\
$\downarrow{f}$ & & $\downarrow{g}$ \\
$\Gamma(O(G_A),u_A)$& $\xrightarrow{\quad \Gamma(\mu) \quad }$ & $\Gamma(O(G_M),u_M)$
\end{tabular}
\end{array}
\end{eqnarray}
By Theorem \ref{ortho}, $\Gamma(O(G_A),u_A)$ and $\Gamma(O(H_A),u_M)$ are orthocompletions of $A$ and $M$, respectively.
Since $\mu$, $\alpha$ and $\beta$ are one-to-one, by Theorem \ref{functor}, $\Gamma(\mu)$, $\Gamma(\alpha)$ and
$\Gamma(\beta)$ are one-to-one.  Since $M$ and $\Gamma(O(G_A),u_A)$ are strongly projectable pseudo $MV$-algebras and $g$ and $\Gamma(\mu)$ are
one-to-one homomorphisms, then $M_1:=\Gamma(\mu)\big(\Gamma(O(G_A),u_A) \big)$ and $M_2:=g(M)$ are strongly projectable pseudo $MV$-subalgebra
of $\Gamma(O(G_M),u_M)$. It follows that $\Gamma (\mu)\circ f (A)\s M_1\cap M_2\s \Gamma(O(G_M),u_M)$.
Since $M_2$ is an essential extension for $\Gamma (\mu)\circ f (A)$ and $\Gamma(O(G_M),u_M)$ is an essential extension for
$M_2$, it can be easily shown that $\Gamma(O(G_M),u_M)$ is an essential extension for $(\Gamma (\mu)\circ f) (A)$, so by Remark \ref{3.9.1},
$M_1\cap M_2$ is a strongly projectable subalgebra of  $\Gamma(O(G_M),u_M)$ containing $(\Gamma (\mu)\circ f) (A)$. Hence, by
the assumption, $M_1\cap M_2=M_2$ (since $M_2$ is a minimal strongly projectable essential extension for $(\Gamma (\mu)\circ f) (A)$)
so, $M_2\s M_1$. It follows that $M\cong M_2\cong (\Gamma(\mu))^{-1}(M_2)\leq \Gamma(O(G_A),u_A)$ is a strongly projectable pseudo $MV$-subalgebra of
$\Gamma(O(G_A),u_A)$  and so $(\Gamma(\mu))^{-1}(M_2)=D$ (since $M$ is a minimal strongly projectable essential extension for $A$).
Therefore, $D\cong M_2\cong M$.
\end{proof}

\section{Conclusion}

In the paper we have studied summand-ideals of a pseudo $MV$-algebra $M$. We have showed that every such an ideal is principal corresponding to a unique Boolean element of $M$. This enables us to define projectable and strongly projectable pseudo $MV$-algebras in a similar way as it was done for $\ell$-groups. Every projectable pseudo $MV$-algebra is representable, i.e., it is a subdirect product of linearly ordered pseudo $MV$-algebras. The main results concern orthocomplete representable pseudo $MV$-algebras and their orthocompletion, Theorem \ref{ortho}. In Theorem \ref{3.11}, it was shown that, for each representable pseudo $MV$-algebra, a minimal strongly projectable essential extension for it does exist.

Since every $MV$-algebra is representable, all results concerning orthocompletion are true also for $MV$-algebras.








\begin{thebibliography}{99}
{\footnotesize

\bibitem{Anderson} M. Anderson and T. Feil, Lattice-Ordered Groups: An Introduction,
{\it Springer Science and Business Media}, USA, 1988.

\bibitem{Ball} R.N. Ball, The  generalized  orthocompletion  and  strongly  projectable  hull  of a  lattice  ordered  group,
{\it Canad.  J. Math.} {\bf 34}  (1982), 621--661.

\bibitem{Bernau} S. Bernau, Orthocompletion of a lattice-ordered groups, {\it Proc. London Math. Soc.} {\bf 16} (1966), 107--130.


\bibitem{Bleier} R. Bleier, The orthocompletion of a lattice-ordered group,
{\it Indaga. Math. (Proc.)} {\bf 79} (1976), 1--7.



\bibitem{Cha}
C.C. Chang,  Algebraic analysis of many valued logics, {\it Trans. Amer. Math. Soc.} {\bf 88} (1958),  467--490.

\bibitem{Cignoli} R. Cignoli, The poset of prime $\ell$-ideals of an Abelian $\ell$-group with a strong unit, {\it J. Algebra} {\bf 184} (1996), 604--612.

\bibitem{mv-book} R. Cignoli, I.M.L. D'Ottaviano and D. Mundici, Algebraic Foundations of Many-valued Reasoning,
{\it Kluwer Acad. Publ., Dordrecht}, 2000.

\bibitem{Conrad2} P. Conrad,  The lateral completion of a lattice-ordered group,
{\it Proc. London Math. Soc.} {\bf 19} (1969), 444--80.

\bibitem{Conrad1} P. Conrad, The hulls of representable $\ell$-groups and $f$-rings,
{\it J. Austral. Math. Soc.} {\bf 16} (1973), 385--415.

\bibitem{Darnel} M. Darnel, Theory of Lattice-Ordered Groups, {\it CRC Press}, USA, 1994.

\bibitem{DvuS}
A. Dvure\v censkij,    States on pseudo MV-algebras,
{\it Studia Logica} {\bf 68}  (2001), 301--327.

\bibitem{Dvu1} A. Dvure\v{c}enskij, On pseudo $MV$-algebras,
{\it Soft Comput.} {\bf 5} (2001), 347--354.

\bibitem{151}
A. Dvure\v censkij, Pseudo MV-algebras
are intervals in $\ell$-groups, {\it J. Austral. Math. Soc.} {\bf 72} (2002), 427--445.

\bibitem{Georgescu} G. Georgescu and A. Iorgulescu, Pseudo $MV$-algebras,
{\it Multiple-Valued Logics} {\bf 6} (2001), 193--215.

\bibitem{Haj} P. H\'ajek,  Observations on
non-commutative fuzzy logic, {\it Soft Computing} {\bf 8} (2003),
38--43.

\bibitem{Jakubik} J. Jakub\' ik, Retract mapping of projectable $MV$-algebras, {\it Soft Comput.} {\bf 4} (2000), 27--32.

\bibitem{Jak2001} J. Jakub\' ik, Direct product decomposition of pseudo MV-algebras, {\it Arch. Math.} {\bf 37} (2001), 131--142.

\bibitem{LPT} A. Ledda, F. Paoli, C. Tsinakis, Lattice-theoretic properties of algebras of logic,
{\it J. Pure  Appl. Algebra} {\bf 218} (2014), 1932--1952.

\bibitem{Mun} D. Mundici,  Interpretation
of AF $C^*$-algebras in \L ukasiewicz sentential calculus, {\it J.
Funct. Anal.} {\bf 65} (1986),  15--63.

\bibitem{Rac} J. Rach\r{u}nek,  A
non-commutative generalization of MV-algebras, {\it Czechoslovak Math.
J.} {\bf 52} (2002), 255--273.

}
\end{thebibliography}
\end{document}